\documentclass[12pt]{amsart}

\usepackage{amssymb}

\makeatletter
\@addtoreset{equation}{section}
\makeatother

\theoremstyle{plain}
\newtheorem{theorem}{Theorem}[section]
\newtheorem{lemma}[theorem]{Lemma}

\theoremstyle{definition}
\newtheorem{definition}[theorem]{Definition}
\newtheorem{remark}[theorem]{Remark}

\newcommand{\Cee}{{\mathbb C}}
\newcommand{\Ree}{{\mathbb R}}
\newcommand{\En}{{\mathbb N}}

\newcommand{\alp}{\alpha}
\newcommand{\del}{\delta}
\newcommand{\Del}{\Delta}
\newcommand{\eps}{\varepsilon}

\newcommand{\lam}{\lambda}

\newcommand{\aaa}{{\mathfrak a}}

\newcommand{\gee}{{\mathfrak g}}
\newcommand{\kay}{{\mathfrak k}}
\newcommand{\pee}{{\mathfrak p}}
\newcommand{\ess}{{\mathfrak s}}

\newcommand{\fW}{{\mathcal W}}

\newcommand{\Ss}{\mathrm{S}}
\newcommand{\SL}{\mathrm{SL}}
\newcommand{\SU}{\mathrm{SU}}
\newcommand{\SO}{\mathrm{SO}}
\newcommand{\Sp}{\mathrm{Sp}}
\newcommand{\Un}{\mathrm{U}}
\newcommand{\Ad}{\mathrm{Ad}}
\newcommand{\ad}{\mathrm{ad}}
\newcommand{\rank}{\mathrm{rank}\,}

\begin{document}

 \title[On $L^1$-$L^2$ dichotomy]{On $L^1$-$L^2$ dichotomy for flat symmetric spaces}
\author{Sanjiv Kumar Gupta}
\address{Dept. of Mathematics \\
Sultan Qaboos University\\
P.O.Box 36 Al Khodh 123\\
Sultanate of Oman}
\email{gupta@squ.edu.om}
\author{Nico Spronk}
\address{Dept. of Pure Mathematics\\
University of Waterloo\\
Waterloo, Ont.,~Canada\\
N2L 3G1}
\email{nico.spronk@uwaterloo.ca}
\thanks{This research is supported in part by an NSERC Discovery Grant and by Sultan
Qaboos University. The first author thank University of Waterloo for their hospitality
when this research was done. }
\thanks{This paper is in final form and no version of it will be submitted
for publication elsewhere.}
\maketitle

\begin{abstract}
 For rank 1 flat symmetric spaces, continuous orbital measures admit absolutely continuous
convolution squares, except for Cartan type AI.  Hence $L^1$-$L^2$ dichotomy for these spaces
holds true in parallel to the compact and non-compact rank 1 symmetric spaces.  We also
study $L^1$-$L^2$ dichotomy for flat symmetric spaces of ranks $p=2,3$ of type AIII, i.e.\
associated with $\SU(p,q)/\Ss(\Un(p)\times\Un(q))$ where $q\geq p$.  For continuous orbital measures given by regular points $L^1$-$L^2$ dichotomy holds.  
We study such measures given by certain singular points when $p=2$, and show that
$L^1$-$L^2$ dichotomy fails.  This is the first time such results are observed for any type
of symmetric spaces of rank 2.
\end{abstract}

\section{Introduction}

We begin by defining certain related symmetric spaces.  Let $G$ be a real, connected,
non-compact semisimple Lie group with finite centre, $K$ a maximal compact subgroup of $G$, and
$\theta$ a Cartan involution on $G$ whose fixed point set is $K$.    Then $G/K$ denotes the associated {\em symmetric space}, and we assume this to be irreducible.
Thanks to \cite[p.\ 154]{Wolf}, we have in
the measure algebra $M(G)$ of $G$ that $M(K\backslash G/K)=m_K\ast M(G)\ast m_K$
 is a commutative subalgebra, where $m_K$ is the $K$-invariant
probability measure supported on $K$.

We consider two related Lie groups.  We let $\gee$ and $\kay$
denote the respective Lie algebras of $G$ and $K$.
Then the derivative of $\theta$ is an involution
on $\gee$ with eigenvalues $\pm 1$, where $\kay$ the eigenspace corresponding to $1$.
We let $\pee$ denote the eigenspace of $-1$, and we get a {\em Cartan decomposition}
$\gee=\kay\oplus\pee$.  We note that $\pee$ is a $\kay$-invariant subspace
of $\gee$ for which $[\pee,\pee]\subseteq\kay$.
First, in the complexification $\gee^\Cee=\gee+i\gee$ we
have that $\gee_c=\kay+i\pee$ is a compact Lie subalgebra giving rise to a compact
group $G_c$, which may first be may be realized in the simply connected covering group,
and projected back into the complexification $G^\Cee$.
Then it is well-known that $G_c/K$ is a compact symmetric space
dual to the symmetric space $G/K$, and is considered to be its compact form.
Secondly, the $\kay$-invariance of $\pee$ entails
that $\pee$ is $\Ad(K)$-invariant, and we consider the semidirect product
group $G_0=K\ltimes\pee$.  Then $G_0/K=\pee$ is the {\em flat} or {\em Euclidean}
symmetric space.

We review orbital measures.  Let $a\in G$.  We define
\[
\nu_a=m_K\ast\delta_a\ast m_K\text{ in }M(K\backslash G/K)
\]
so
\[
\langle \nu_a,f\rangle=\int_G f\,d\nu_a
=\int_K\int_K f(k_1ak_2)\,dk_1\,dk_2
\]
for $f$ in $C_c(G)$, where we use Haar probability measure on $K$.
This measure is supported on $KaK$, which is the orbit of $a$ under the
action $((k_1,k_2),a)=k_1ak_2^{-1}$ of $K\times K$ on $G$.  This measure
is singular as the manifold $KaK$ has lower dimension than $G$.
Thanks to Cartan decomposition of the group, $G=KAK$, we may assume that
$a=\exp(H)$ where $H\in \aaa$.  Here $\aaa$ is a maximal abelian subalgebra
of $\gee$, contained in the subspace $\pee$.
Similarly, for $a_c$ in $G_c$ we define
$\nu_{a_c}$ in $M(K\backslash G_c/K)$.

For $(k,H)$ in $G_0$ we can similarly define
$\nu_{(k,H)}=\nu_{(1,H)}$.  However, it is better to consider a measure on $\pee$, as follows.
For $H\in \pee$  we consider the orbit
\[
O_H=\{\Ad(k)H:k\in K\}
\]
and we define the {\em orbital measure} $\mu_H$ on $\pee$ by
\begin{equation}\label{eq:orb_meas}
\langle \mu_H,f\rangle=\int_\pee f\,d\mu_H=\int_K f(\Ad(k)H)\,dk
\end{equation}
for $f$ in $C_c(\pee)$; really for $f$ in $C(\pee)$, as this measure is supported on the compact
orbit $O_H$.  As $\dim O_H\leq \dim\pee-\dim\aaa$, this is always a singular measure.
Notice that for $H$ in $\pee$ and $k$ in $K$ we have in $M(G_0)$ that
\[
\nu_{(1,H)}=m_K\ast\delta_{(1,H)}\ast m_K=m_K\ast\del_{(k,0)(1,H)(k^{-1},0)}\ast m_K
=\nu_{(1,\Ad(k)H)}
\]
showing that $\nu_{(1,H)}\mapsto \mu_H$ is well-defined and injective.  Furthermore it is standard
to compute that if $X,Y\in\pee$ then
\[
\nu_{(1,X)}\ast\nu_{(1,Y)}
=m_K\ast\left(\int_K\int_K  \delta_{(1,\Ad(k_1)X+\Ad(k_2)Y)} \,dk_1\,dk_2\right)\ast m_K
\]
in $M(G_0)$, while $\mu_X\ast\mu_Y=\int_K\int_K \delta_{\Ad(k_1)X+\Ad(k_2)Y}\,dk_1\,dk_2$
in $M(\pee)$.  This shows that $M(K\backslash G_0/K)$ is isomorphic to the algebra
$M_K(\pee)$ of orbital measures.  Furthermore, this isomorphism takes $L^1(K\backslash G_0/K)$
onto $L^1_K(\pee)$.  This can be studied in the language of hypergroups; see
\cite{BH}.

For $\mu$ in $M(G),\,M(G_c)$, or $M(\pee)$, let
$\mu^{\ast r}=\mu\ast\dots\ast\mu$ ($r$-fold convolution).

\begin{definition}\label{def:L1-L2-dic}
\begin{enumerate}
\item The symmetric space $G/K$ satisfies {\em $L^1$-$L^2$ dichotomy} if for
every $a$ in $G$ and $r$ in $\En$, $\nu_a^{\ast r}$ is either singular, or is in
$L^2(G)$.
\item The symmetric space $G_c/K$ satisfies {\em $L^1$-$L^2$ dichotomy} if for
every $a_c$ in $G_c$ and $r$ in $\En$, $\nu_{a_c}^{\ast r}$ is  either singular, or is in
$L^2(G_c)$.
\item The symmetric space $G_0/K$ satisfies {\em $L^1$-$L^2$ dichotomy} if for
every $H$ in $\pee$ and $r$ in $\En$,  $\mu_H^{\ast r}$ singular, or is in
$L^2(\pee)$.
\end{enumerate}
\end{definition}

\subsection{History.}
The roots of problems of the sort in which we are interested
go back to Dunkl~\cite{Du}. He showed in
 $M(\SO(n)\backslash\SO(n+1)/\SO(n))\cong M_{\SO(n)}(\mathbb{S}^n)$, where
 $\cong$ indicates isomorphism as convolution algebras, that
convolution products of continuous elements are in  $L^1(\SO(n+1))$.
Ragozin~\cite{Ra1} showed for a
simple compact Lie group $S$, with $n=\dim S$, that $\mu^{\ast n}\in L^1(S)$ for each 
continuous central measure $\mu$.  Notice that the algebra of continuous central measures
is isomorphic to $M(\Del\backslash S\times S/\Del)$
where $\Del=\{(s,s):s\in S\}$.
In \cite{Ra2}, the case of $\nu_a$ for a symmetric space
$G/K$ as above is studied, and it is shown that for $n=\dim(G/K)$ that $\nu_a^{\ast n}\in L^1(G)$.
For central measures on $S$, the number of convolution powers required to land in $L^1(S)$
has been sharpened considerably; see \cite{GH_GAFA}.  For the symmetric space $G/K$
where $G$ is simple, the results
of \cite{Ra2} are sharpened by Graczyk and Sawyer \cite{GSJFA}
(see the survey \cite{GSSurvey}) and made completely sharp
by Gupta and Hare \cite{GHJMAA17}.


Gupta and Hare~\cite{GHAdv}
first discovered the $L^1$-$L^2$ dichotomy for orbital measures on $S$.
Since then, many related results followed, including \cite{HH1,HJSY,Wr,GHboll,GHjlt21}.
In particular, the $L^1$-$L^2$ dichotomy always holds for orbital measures
in $M_S(\ess)\cong M(S\backslash S\ltimes\ess/S)$
($\ess$ is the Lie algebra of a simple compact Lie group $S$, acted upon by adjoint action),
as shown in classical cases in \cite{GHS}, and then generally by Hare and He~\cite{HH2}.
In \cite{AGPsu2so2} (also \cite{HH}), and \cite{GHaustra2022}
it is proved that $L^1$-$L^2$ dichotomy fails
in $\SU(2)/\SO(2)$ and $\SL(2,\Ree)/\SO(2)$, respectively.

\subsection{Present results.} We study the $L^1$-$L^2$ dichotomy for certain flat symmetric
spaces, $G_0/K$, i.e.\ for orbital measures
defined in (\ref{eq:orb_meas}).

We succeed fully in the rank 1 setting, showing that
$L^1$-$L^2$ dichotomy holds, except for $G_0/K=\SL(2,\Ree)_0/\SO(2)$ where it fails.
Notice that this coincides with the results about rank 1 spaces $G_c/K$ and $G/K$, noted above.

 We study cases
 $G_0/K=\SU(p,q)_0/\Ss(\Un(p)\times \Un(q))$ for $p=2,3$, $q\geq p$.  For $p=2$ the
 $L^1$-$L^2$ dichotomy holds for regular $H$, but the singular points prove more interesting.
 For convenience, we classify singular points as being either type D or type A.  For type D
 we obtain $L^1$-$L^2$ dichotomy when $q>2$, but fails if $q=2$.  For type A points
 we learn failure of $L^1$-$L^2$ dichotomy for $q\geq 3$, with increasingly sharp bounds
 for increasing $q$.  However, we cannot settle the case $q=2$.  For $p=3$ we obtain
 $L^1$-$L^2$ dichotomy for regular points, and indicate the limitations of our methods.
 It would be extremely interesting to learn if there are analogous results to ours
 for the $L^1$-$L^2$ dichotomy in, for example, $\SU(2,q)/\Ss(\Un(2)\times \Un(q))$ and/or
 $\SU(2+q)/\Ss(\Un(2)\times \Un(q))$.

Our method is roughly as follows.  We can represent
the spherical transform of the convolution power of an orbital measure $\mu_H^{\ast n}$
in terms of a basic spherical function.  For certain symmetric spaces,
amongst those which the spherical functions are given explicitly by
Ben Sa\"{\i}d and {\O}rsted \cite{BSO}, we can then
test what powers of the basic spherical functions are square integrable using asymptotic
estimates with Bessel functions.
The Plancherel formula for orbital measures then implies
the desired result.


\section{Notation and basic facts}

\subsection{Lie algebra set-up} Let $G$ be a connected, real, non-compact semisimple Lie group
with finite centre, and $K$ is maximal compact subgroup fixed by the Cartan involution.
Let $\gee=\kay\oplus\pee$ be the corresponding Cartan decomposition of the group's
Lie algebra, where $\kay$ is the Lie algebra of $K$ and $\pee$ its orthogonal complement
with respect to the Killing form $B$ on $\gee$.  We assume that the symmetric space
$G/K$ is irreducible.

We follow \cite[Chapter IX]{He}.
Let $\aaa$ be a subspace of $\pee$ which is maximally abelian as a subalgebra of $\gee$.
Any two such subspaces are conjugate, and in fact we have $\pee=\bigcup_{k\in K}\Ad(k)\aaa$,
showing that each $O_H\cap\aaa\not=\varnothing$ for $H$ in $\pee$.
The {\em rank} is defined as
\[
\rank G/K=\dim\aaa.
\]
For $\alp$ in $\aaa^*$ we let
\[
\gee_\alp=\{Y\in \gee:\ad(X)Y=[X,Y]=\alp(X)Y\text{ for all }X\text{ in }\aaa\}
\]
and let the restricted roots be given by
\[
\Phi=\{\alp\in\aaa^*\setminus\{0\}:\gee_\alp\not=\{0\}\}
\]
where we note that $\aaa=\gee_0\cap\pee$.
We consider the set
\[
\aaa'=\aaa\setminus\bigcup_{\alp\in\Phi}\ker\alp
\]
of {\em regular} elements
and fix a connected component $\aaa'$, which is called a Weyl chamber and denoted by $\aaa^+$.
Any two of these Weyl chambers may be interchanged via the action of the Weyl group
$W\subset K$. We then consider the positive restricted roots
\[
\Phi^+=\{\alp\in\Phi:\alp(X)>0\text{ for all }X\text{ in }\aaa^+\}.
\]
We let the {\em multiplicity} of a positive restricted root be given by
\[
m_\alp=\dim\gee_\alp\text{ for }\alp\text{ in }\Phi^+.
\]
The Killing form $B$ is an inner product on $\pee$, hence on $\aaa$, and thus for each $\lam$ in
$\aaa^*$ there is $E_\lam$ in $\aaa$ for which $\lam=B(E_\lam,\cdot)$.

\subsection{Spherical transform on flat symmetric space}
Recall that $G_0=K\ltimes\pee$ so $G_0/K\cong\pee$.
Given $\lam$ in $\aaa^*$, the {\em spherical function} on $\pee$ is given
as in \cite[Proposition 4.8]{Hegroups} by
\begin{equation}\label{eq:sph_func}
\psi_\lam(Y)=\int_K e^{iB(E_\lam,\Ad(k)Y)}\,dk.
\end{equation}
Then for any compactly supported $\nu$ in $M(\pee)$ we let its {\em spherical transform}
\begin{equation}\label{eq:sph_trans}
\hat{\nu}(\lam)=\int_\pee\psi_{-\lam}(Y)\,d\nu(Y).
\end{equation}
We shall use the following.  Recall that the orbital measure $\mu_H$ is defined in (\ref{eq:orb_meas}).

\begin{lemma}\label{lem:sph_trans_muH}
Let $H_1,\dots,H_r\in \pee$ and $\rho=\mu_{H_1}\ast\dots\ast\mu_{H_r}$.  Then
\[
\hat{\rho}(\lam)=\prod_{i=1}^r\psi_{-\lam}(H_i)\text{ for }\lam\text{ in }\aaa^*.
\]
\end{lemma}

\begin{proof}
Let $M_K(\pee)$ denote the space of orbital measures, i.e.\ $\mu$ for which
$\int_\pee f\,d\mu=\int_\pee f(\Ad(k)X)\,d\mu(X)$ for each $f$ in $C_c(\pee)$
and $k$ in $K$.  Notice that each $\mu_H\in M_K(\pee)$.  Also,
if $\lam\in\pee^*$ then $\psi_\lam(\Ad(k)Y)=\psi_\lam(Y)$ for $k$ in $K$ and
$Y\in \pee$, so
\[
\widehat{\mu_H}(\lam)=\int_K\psi_{-\lam}(\Ad(k)H)\,dk=\int_K\psi_{-\lam}(H)\,dk=\psi_{-\lam}(H).
\]
If $\mu$ and $\nu$ are two compactly supported elements of $M_K(\pee)$ then for
$\lam$ in $\pee^*$ we have
\begin{align*}
&\widehat{\mu\ast\nu}(\lam)
=\int_\pee\int_\pee \int_K e^{iB(E_{-\lam},\Ad(k)(X+Y))}\,d\mu(X)\,d\nu(Y)\,dk \\
&\;=\int_K\int_\pee e^{iB(E_{-\lam},\Ad(k)X)}\,d\mu(X)
\int_\pee e^{iB(E_{-\lam},\Ad(k)Y)}\,d\nu(Y)\,dk \\
&\;=\int_K\int_\pee e^{iB(E_{-\lam},\Ad(kk')X)}\,d\mu(X)
\int_\pee\int_K e^{iB(E_{-\lam},\Ad(kk')Y)}\,dk'\,d\nu(Y)\,dk \\
&\;=\int_\pee\int_K e^{iB(E_{-\lam},\Ad(kk')X)}\,d\mu(X)\,dk
\int_\pee \int_Ke^{iB(E_{-\lam},\Ad(k')Y)}\,dk'\,d\nu(Y) \\
&\;=\hat{\mu}(\lam)\hat{\nu}(\lam)
\end{align*}
where we have used Fubini's theorem, and probability Haar measure on $K$.
The desired result follows from combining these two formulas.
\end{proof}

We require the Plancherel formula below.

\begin{lemma}\label{lem:plancherel}
Let $\mu$ be an $\Ad(K)$-invariant measure on $\pee$.  Then $\mu\in L^2(\pee)$
if and only if
\[
\|\mu\|_{L^2(\pee)}^2=c\int_{\aaa^\ast_+}|\hat{\mu}(\lam)|^2\del(\lam)\,d\lam<\infty
\]
where $\displaystyle \del(\lam)=\left|\prod_{\alp\in\Phi^+} \alp(E_\lam)^{m_\alp}\right|$,
$\aaa^*_+=\{\lam\in \aaa^*:E_\lam\in\aaa^+\}$,
and $c>0$ is a constant.
\end{lemma}

\begin{proof}
We have that \cite[Chapter IV, Theorem 9.1]{Hegroups} gives
\[
\|\mu\|_{L^2(\pee)}^2=C\int_{\aaa^*}|\hat{\mu}(\lam)|^2\del(\lam)\,d\lam<\infty
\]
for a constant $C>0$.  Since $B$ is $\Ad(K)$-invariant,
hence $\Ad(W)$ invariant, and $W$ interchanges Weyl chambers,
the adjoint action of $W$ interchanges dual Weyl chambers $\aaa^*_+$.
Thus
\[
\int_{\aaa^*}|\hat{\mu}(\lam)|^2\del(\lam)\,d\lam
=\frac{|W|}{|N|}\int_{\aaa^*_+}|\hat{\mu}(\lam)|^2\del(\lam)\,d\lam
\]
where $N=\{w\in W:w(\aaa^+)=\aaa^+\}$.
\end{proof}

Let us also note the following result regarding non-singularity of product measures.

\begin{lemma}\label{lem:inL1}
Suppose $G$ is a simple Lie group.
\begin{enumerate}
\item Let $H$ be an element of $\pee$ which is regular in the sense that it is conjugate
to a regular element of $\aaa$.  Then $\mu_H^{\ast 2}\in L^1(\pee)$.
\item Let $H\in \pee\setminus\{0\}$.  If $\rank G/K\geq 2$, then $\mu_H^{\ast k(G)}\in L^1(\pee)$ where
\[
k(G)=\begin{cases} 
\rank G/K+1 &\text{if }\Phi^+\text{ is of type A$_n$ or D$_3$} \\
\rank G/K &\text{otherwise.}\end{cases}
\]
\end{enumerate}
\end{lemma}

\begin{proof}
(1) Without loss of generality, $H\in\aaa'$.
Then by the proof of \cite[Theorem 4.3]{AG}, $\mu_H^{\ast 2}\in L^1(\pee)$.

(2) Without loss of generality, $H\in\aaa$.  Then by \cite[Corollary 7]{GSJFA},
$\nu_{\exp(H)}^{\ast (\rank G/K+1)}\in L^1(G)$.  This estimate is improved
in \cite[Corollary 3]{GHJMAA17} in cases complementary to A$_n$ and D$_3$.
(Though not explicitly stated in \cite{GHJMAA17}, the result is valid for $\rank G/K\geq 2$.)  Then,
we apply \cite[Theorem 3.1]{AG} to get the desired result.
\end{proof}

\section{Rank 1 case}

We obtain a very complete answer in the rank 1 case.  Let us begin by listing
the non-compact irreducible rank 1 symmetric spaces $G/K$, their restricted root
systems $\Phi$, and multiplicities of members of $\Phi^+$.
We are using information from \cite[p. 518]{He}, \cite[Appendix]{HH}, and, for the FII case,
\cite[11.2.3]{Wolf}.

\begin{center}
\begin{tabular}{lllll}\hline
& space & $\Phi^+$  & $m_\alp$ & $m_{2\alp}$ \\ \hline
AI & $\SL(2,\Ree)/\SO(2)$ & A$_1$ & 1 & $-$ \\
AII & $\SL(2,\mathbb{H})/\Sp(2)$ & A$_1$ & 4 & $-$ \\
AIII & $\SU(1,q)/\Ss(\Un(1)\times \Un(q))$ & BC$_1$ & $2(q-1)$ & 1 \\
BDI & $\SO_0(1,q)/1\times \SO(q)$ & A$_1$ & $q-1$ & $-$ \\
CII & $\Sp(1,q)/\Sp(1)\times\Sp(q)$ & BC$_1$ & $4(q-1)$ & 3 \\
FII & $F_{4,B_4}/\mathrm{Spin}(9)$ & BC$_1$ & 8 & 7 \\ \hline
\end{tabular}
\end{center}
The leftmost labels are from Cartan's classification of symmetric spaces, and we always have
$q\geq 2$.  Note that if $q=2$ then case BDI is isomorphic to AI.  Also, $\SL(2,\mathbb{H})$
is denoted $\SU^*(4)$ in some texts, and
$F_{4,B_4}$ denotes the non-compact real form of the exceptional compact Lie group $F_4$.

\begin{theorem}\label{theo:rank1}
Let $G/K$ be a non-compact, rank 1, irreducible symmetric space.
Then the flat symmetric
space $G_0/K$ satisfies the $L^1$-$L^2$ dichotomy except for the case
AI (which is BDI with $q=2$).
\end{theorem}

\begin{proof}
All rank 1 symmetric spaces arise from simple Lie groups.
If $H\in\pee\setminus\{0\}$ then $H$ is a regular element of $\aaa=\Ree H$ and thus
Lemma \ref{lem:inL1} provides that $\mu_H^{\ast 2}\in L^1(\pee)$.

We may write $H=tL$ for $t$ in $\Ree^+$ where $L$ is suitably normalized, so according to
\cite[Theorem 7.2]{BSO} we have
\begin{equation}
\phi_{\lam}(tL)=C(t\lam)^{-(m_\alp+m_{2\alp}-1)/2}J_{(m_\alp+m_{2\alp}-1)/2}(t\lam).
\label{eq:psi1}
\end{equation}
where $J_\nu$ is a Bessel function of the first kind and $C>0$ is a constant, independent
of $t$ and $\lam$.  Then, thanks to Lemmas \ref{lem:sph_trans_muH} and \ref{lem:plancherel},
our goal is to find $k>0$ for which
\begin{equation}
\int_0^\infty |\phi_{-\lam}(tL)|^{2k} \lam^{m_\alp+m_{2\alp}}\,d\lam<\infty \label{eq:int1}
\end{equation}
since $\aaa^+\cong\Ree^+$.  We let $\nu=(m_\alp+m_{2\alp}-1)/2$ so
\[
\nu\begin{cases}
=0 &\text{in case AI (hence BDI if }q=2) \\
>0 &\text{otherwise.}\end{cases}
\]
Using (\ref{eq:psi1}), we get that the formula for the integrand
in (\ref{eq:int1}) is proportional to
\[
|\lam^{-\nu}J_\nu(\lam t)|^{2k}\lam^{2\nu+1}=|J_\nu(\lam t)|^{2k}\lam^{2\nu(1-k)+1}
\]
For $\lam>\!\!>1$  we have estimate $|J_\nu(\lam t)|\leq \frac{C}{\sqrt{\lam t}}$,
so the last formula is proportionally bounded above by
\[
\lam^{-k}\lam^{2(1-k)\nu+1}=\lam^{-(2\nu+1)(k-1)}.
\]
Thus we can see that $k\geq 2$ suffices to obtain the boundedness (\ref{eq:int1}), when $\nu>0$,
so we get $L^1$-$L^2$ dichotomy in these cases.
However, for $\nu=0$, $k\geq 3$ suffices.

Let us now consider $k=2$ in the case $\nu=0$.
Here, the integrand in (\ref{eq:int1}) is
\[
|J_0(t\lam)|^{2k}\lam.
\]
We have for $\lam>\!\!>1$ the asymptotic estimate
\[
J_0(t\lam)\approx \sqrt{\frac{2}{\pi t\lam}}\cos\left(t\lam -\frac{\pi}{4}\right).
\]
With $\lam\geq 1$ we consider $\lam$ for which
$t\lam-\frac{\pi}{4}\in[-\frac{\pi}{4}+k\pi,\frac{\pi}{4}+k\pi]$ for $\lam\geq N>\!\!>1$,
so $\cos^4(t\lam-\frac{\pi}{4})\geq \frac{1}{4}$, to
get estimate
\[
\int_0^\infty |J_0(t\lam)|^4\lam\,d\lam
\geq C\sum_{k=N}^\infty \int_{\frac{k\pi}{t}}^{\frac{(2k+1)\pi}{2t}}\frac{d\lam}{\lam}
=\infty.
\]
Hence we deduce that $k\geq 3$ is necessary to obtain (\ref{eq:int1}) for the $\nu=0$ case.
In other words, $L^1$-$L^2$ dichotomy fails in this case.
\end{proof}

\section{Rank 2 and rank 3 special unitary groups}

\subsection{Basic data} We now consider some symmetric spaces of Cartan type AIII:
\[
\SU(p,q)/\Ss(\Un(p)\times \Un(q)), p=2,3,q\geq p.
\]
These always have rank $p$.  The positive roots structures and multiplicities are as follows,
with data taken from \cite[p.\ 72]{CM}.

\begin{center}
\begin{tabular}{ccccc}\hline
$q$ &  $\Phi^+$ & $m_0$ & $m_1$ & $m_2$ \\ \hline
$p$ &  C$_p$ & 2 &  $-$ & 1\\
$>\!p$ & BC$_p$ & 2 & $2(q-p)$ & 1 \\ \hline
\end{tabular}
\end{center}
Here we have
\begin{align*}
\text{C}_p&=\{\alp_i\pm\alp_j,2\alp_k:1\leq i<j\leq p,1\leq k\leq p\} \\
\text{BC}_p&=\{\alp_i\pm\alp_j,\alp_k,2\alp_k:1\leq i<j\leq p,1\leq k\leq p\}
\end{align*}
with $m_0=m_{\alp_i\pm\alp_j},\,m_1=m_{\alp_k},\,m_2=m_{2\alp_k},\,
1\leq i<j\leq p,1\leq k\leq p$.

Taking basis $(\alp_1,\dots,\alp_p)$ for $\aaa^*$ and $(E_{\alp_1},\dots,E_{\alp_p})$
for $\aaa$, \cite[Theorem 6.1]{BSO} provides for $\psi_\lam(X)$
\[
\psi(\lam,X)=c_0\prod_{k=1}^p(x_k\lam_k)\frac{\det[J_{q-p}(x_i\lam_j)]_{i,j=1,\dots,p}}
{\prod_{1\leq i<j\leq p}(x_i^2-x_j^2)(\lam_i^2-\lam_j^2)}
\]
where $J_r$ denotes a Bessel function of first type.
Based on Lemmas \ref{lem:sph_trans_muH} and \ref{lem:plancherel},
we will wish to find $k\geq 1$ for which
\begin{equation}\label{eq:int3}
\int_\fW|\psi(\lam,X)|^{2k}\prod_{1\leq i<j\leq p}(\lam_i^2-\lam_j^2)^2(\lam_1\dots\lam_p)^{2(q-p)+1}
\,d\lam<\infty
\end{equation}
where we integrate over the Weyl chamber which is parametrized by
\[
\fW=\{(\lam_1,\dots,\lam_p):\lam_1>\dots>\lam_p>0\}.
\]
We will find it convenient to set
\begin{equation}\label{eq:fr}
f_r(s)=\frac{J_r(s)}{s^r}\text{ where }r=q-p
\end{equation}
so $f_0=J_0$.  Then the integrand of (\ref{eq:int3}) becomes
\begin{equation}\label{eq:integrand}
\varphi_k(\lam,X)=\frac{|[f_r(x_i\lam_j)]_{i,j=1,\dots,p}|^{2k}}
{\prod_{1\leq i<j\leq p}(x_i^2-x_j^2)^{2k}(\lam_i^2-\lam_j^2)^{2k-2}}(\lam_1\dots\lam_p)^{1+2r}
\end{equation}
where $|A|=|\det(A)|$ for a matrix $A$.

We will require asymptotic estimates for derivatives of the analytic functions $f_r$ of (\ref{eq:fr}).  Since for $s>\!\!>1$ we have
\begin{equation}\label{eq:Jr-asymp}
J_r(s)\approx \sqrt{\frac{2}{\pi s}}\cos(s-\frac{r\pi}{2}-\frac{\pi}{4})
\end{equation}
we obtain the estimate
\begin{equation}\label{eq:fr-est}
|f_r(s)|\leq \frac{C}{s^{r+1/2}}.
\end{equation}
We use that fact that $\frac{d}{ds}[s^rJ_r(s)]=s^rJ_{r-1}(s)$ to see that
for $r>0$
\begin{equation}\label{eq:fr-der}
f_r'(s)=\frac{d}{ds}\left(\frac{s^rJ_r(s)}{s^{2r}}\right)=\frac{J_{r-1}(s)s-2rJ_r(s)}{s^{r+1}}.
\end{equation}
Hence for $s>\!\!>1$ we have
\begin{equation}\label{eq:fr-dir-est}
|f_r'(s)|\leq \sqrt{\frac{2}{\pi s}}\frac{s+2r}{s^{r+1}}\leq \frac{C'}{s^{r+1/2}}
\end{equation}
where $C'>0$ is independent of $s$.
If $r=0$ then $f_0'(s)=J_0'(s)=J_{-1}(s)=-J_1(s)$, so (\ref{eq:fr-dir-est}) still holds.
We use similar reasoning to that in (\ref{eq:fr-der}) to compute
\[
f_r''(s)=\frac{s^2J_{r-2}(s)+(1-4r)sJ_{r-1}(s)+2r(2r+1)J_r(s)}{s^{r+2}}
\]
to get, as in (\ref{eq:fr-dir-est}), for $s>\!\!>1$ the estimate
\begin{equation}\label{eq:fr-2der-est}
|f_r''(s)|\leq \frac{C''}{s^{r+1/2}}
\end{equation}
where $C''>0$ is independent of $s$.

\subsection{Rank 2 regular case}\label{ssec:reg_rank2}
It will be convenient here, and in the next section, to adopt a decomposition
of the domain $\fW$ of (\ref{eq:int3}).  We consider consider a ball $B$ in $\Ree^2$
of large radius (as dictated by asymptotic estimates of Bessel functions) and set
\begin{align*}
\fW_1&=\{(\lam_1,\lam_2)\in \fW\setminus B:\lam_2\geq \tfrac{1}{2}\lam_1\} \\
\fW_2&=\{(\lam_1,\lam_2)\in \fW\setminus B:\lam_2<\tfrac{1}{2}\lam_1\}.
\end{align*}
so that $\fW=(\fW\cap B)\cup \fW_1\cup \fW_2$.

\begin{theorem}\label{theo:int3-reg}
The flat symmetric space $\SU(2,q)_0/\Ss(\Un(2)\times \Un(q))$ ($q\geq 2$) satisfies the
$L^1$-$L^2$ dichotomy at regular points.
\end{theorem}

\begin{proof}
Lemma \ref{lem:inL1} informs us that we need to inspect if (\ref{eq:int3}) holds for
for $k\geq 2$.  Since we are interested in regular points we may assume that
$X=(x_1,x_2)$ where $x_1>x_2>0$, and we are interested in the integrability
of $\varphi_k(\lam,X)$ over which is proportional to
\[
\varphi(\lam_1,\lam_2)=\frac{\begin{vmatrix}f_r(x_1\lam_1) & f_r(x_1\lam_2) \\
f_r(x_2\lam_1) & f_r(x_2\lam_2)\end{vmatrix}^{2k}}{(\lam_1^2-\lam_2^2)^{2k-2}}
(\lam_1\lam_2)^{1+2r}
\]
on $\fW$, where $r=q-2$.

We now estimate on $\fW_1$ where there is potential for $\lam_1-\lam_2$ to become small.
First, using properties of determinants we
see that $\varphi(\lam_1,\lam_2)$ has form
\[
\begin{vmatrix}f_r(x_1\lam_1) & f_r(x_1\lam_2) \\
f_r(x_2\lam_1) & f_r(x_2\lam_2)\end{vmatrix}^2
\begin{vmatrix}\frac{f_r(x_1\lam_1)-f_r(x_1\lam_2)}{\lam_1-\lam_2}
& f_r(x_1\lam_2) \\
 \frac{f_r(x_2\lam_1)-f_r(x_2\lam_2)}{\lam_1-\lam_2}
 & f_r(x_2\lam_2)\end{vmatrix}^{2k-2}
\frac{(\lam_1\lam_2)^{1+2r}}{(\lam_1+\lam_2)^{2k-2}}.
\]
We then apply mean value theorem to turn the middle term into
\[
\begin{vmatrix} x_1f_r'(x_1\lam_{11}) & f_r(x_2\lam_1) \\
x_2f_r'(x_2\lam_{12}) & f_r(x_2\lam_2)\end{vmatrix}^{2k-2}
\]
where $\lam_1>\lam_{1j}>\lam_2$ for $j=1,2$.
We employ the estimates (\ref{eq:fr-est}) and (\ref{eq:fr-dir-est}),
and then the fact that $\lam_1>\lam_2\geq\frac{1}{2}\lam_1$ in $W_1$, to see that
$\varphi(\lam_1,\lam_2)$ is bounded above by
\begin{align*}
&\leq \frac{C'}{(\lam_1\lam_2)^{2r+1}}\left(\frac{1}{(\lam_{11}\lam_2)^{r+1/2}}
+\frac{1}{(\lam_1\lam_{12})^{r+1/2}}\right)^{2k-2}
\frac{(\lam_1\lam_2)^{1+2r}}{(\lam_1+\lam_2)^{2k-2}} \\
&\leq  \frac{C''\lam_1^{2(2r+1)}}{\lam_1^{2(2r+1)}\lam_1^{(2r+1)(k-1)}\lam^{2k-2}}
=\frac{C''}{\lam_1^{(2r+1)(k-1)+2k-2}}
\end{align*}
Hence we get we use polar coordinates $(\lam_1,\lam_2)=(\rho\cos\theta,
\rho\sin\theta)$, and the fact that $\cos\theta\geq \frac{1}{\sqrt{2}}$ on $[0,\frac{\pi}{4}]$
to see that if $k\geq 2$ then
\begin{equation}\label{eq:int_W1}
\int_{\fW_1}\varphi(\lam_1,\lam_2)\,d(\lam_1,\lam_2)
\leq K\int_{\pi/8}^{\pi/4}\int_1^\infty \frac{\rho\,d\rho\,d\theta}{\rho^{(2r+1)(k-1)+2k-2}}<\infty.
\end{equation}

We now estimate on $\fW_2$.
We use (\ref{eq:fr-est}) on variable $\lam_1$, the fact that each $f_r(x_k\lam_2)$ is bounded,
and then the fact that  $\frac{1}{2}\lam_1>\lam_2>0$ so
$\lam_1^2-\lam_2^2> \frac{3}{4}\lam_1^2$, to see that
\begin{align*}
\varphi(\lam_1,\lam_2)
&\leq C'\left(\frac{1}{\lam_1^{r+1/2}}+\frac{1}{\lam_1^{r+1/2}}\right)^{2k}
\frac{(\lam_1\lam_2)^{1+2r}}{(\lam_1^2-\lam_2^2)^{2k-2}} \\
&\leq \frac{C''\lam_1^{2(1+2r)}}{\lam_1^{k(1+2r)}\lam_1^{2(2k-2)}}
=\frac{C''}{\lam_1^{(k-2)(1+2r)+4k-4}}.
\end{align*}
As above, we use polar coordinates to obtain for $k\geq 2$ that
\begin{equation}\label{eq:int_W2}
\int_{\fW_2}\varphi(\lam_1,\lam_2)\,d(\lam_1,\lam_2)
\leq K\int_0^{\pi/8}\int_1^\infty \frac{\rho\,d\rho\,d\theta}{\rho^{(k-2)(1+2r)+4k-4}}<\infty.
\end{equation}

Then (\ref{eq:int_W1}) and (\ref{eq:int_W2}) implies (\ref{eq:int3}) for $k\geq 2$.
\end{proof}

\subsection{Rank 2 singular cases} Here we get distinctive behaviour depending
upon what type of singular point we are considering.  If $H\cong (x_1,x_2)$
with respect to the basis $(E_{\alp_1},E_{\alp_2})$, and with $(x_1,x_2)\in\overline{\fW}$,
we say that
\begin{itemize}
\item $H$ is singular of type D (diagonal) if $x_1=x_2=x>0$, and
\item $H$ is singular of type A (axis) if $x_1=x>x_2=0$.
\end{itemize}

\begin{theorem}\label{theo:singular1}
Let $H$ be a singular point of type D in the flat symmetric space
$\SU(2,q)_0/\Ss(\Un(2)\times \Un(q))$.  Then $L^1$-$L^2$ dichotomy holds for
$\mu_H$ if $q>2$, and fails if $q=2$.
\end{theorem}

\begin{proof}
Lemma \ref{lem:inL1} always provides that $\mu_H^{\ast 2}\in L^1(\pee)$.
Hence it remains to check for which $k\geq 2$ (\ref{eq:int3}) holds.
We first observe that $\varphi_k(\lam,X)$ is proportional to
\begin{align*}
\frac{\begin{vmatrix}f_r(x_1\lam_1) & f_r(x_1\lam_2) \\
f_r(x_2\lam_1) & f_r(x_2\lam_2)\end{vmatrix}}
{(\lam_1^2-\lam_2^2)(x_1-x_2)}
&=\frac{\begin{vmatrix}\frac{f_r(x_1\lam_1)-f_r(x_2\lam_1)}{x_1-x_2}
 & \frac{f_r(x_1\lam_2)-f_r(x_2\lam_2)}{x_1-x_2} \\
f_r(x_2\lam_1) & f_r(x_2\lam_2)\end{vmatrix}}
{\lam_1^2-\lam_2^2} \\
& \overset{x_1,x_2\to x}{\xrightarrow{\hspace{1.4cm}}}
\frac{\begin{vmatrix}\lam_1f'_r(x\lam_1) & \lam_2 f'_r(x\lam_2) \\
f_r(x\lam_1) & f_r(x\lam_2)\end{vmatrix}}
{\lam_1^2-\lam_2^2}
\end{align*}
and hence we are interested in the integrability of
\[
\varphi(\lam_1,\lam_2)=\frac{\begin{vmatrix}\lam_1f'_r(x\lam_1) & \lam_2 f'_r(x\lam_2) \\
f_r(x\lam_1) & f_r(x\lam_2)\end{vmatrix}^{2k}}
{(\lam_1^2-\lam_2^2)^{2k-2}}(\lam_1\lam_2)^{1+2r}
\]
over $\fW$, where $r=q-2$.

Using similar reasoning as in proof of Theorem \ref{theo:int3-reg} we proceed as follows
for estimates on $\fW_1$.
We have that $\varphi(\lam_1,\lam_2)$ is given by the expression
\[
\begin{vmatrix}\lam_1f'_r(x\lam_1) & \lam_2 f'_r(x\lam_2) \\
f_r(x\lam_1) & f_r(x\lam_2)\end{vmatrix}^2
\begin{vmatrix}\frac{\lam_1f'_r(x\lam_1)- \lam_2 f'_r(x\lam_2)}{\lam_1-\lam_2}
& \lam_2 f'_r(x\lam_2) \\
\frac{f_r(x\lam_1)-f_r(x\lam_2)}{\lam_1-\lam_2}
& f_r(x\lam_2)\end{vmatrix}^{2k-2}\frac{(\lam_1\lam_2)^{1+2r}}{(\lam_1+\lam_2)^{2k-2}}
\]
where, since $\frac{d}{d\lam}[\lam f'(x\lam)]=f'(x\lam)+x\lam f''_r(x\lam)$, the middle term becomes
\[
\begin{vmatrix} f'_r(x\lam_{11})+x\lam_{11}f''_r(x\lam_{11})
& \lam_2 f'_r(x\lam_2) \\
xf_r'(x\lam_{12}) &  f_r(x\lam_2)\end{vmatrix}^{2k-2}
\]
where $\lam_1>\lam_{1j}>\lam_2$ for $j=1,2$.  We then use (\ref{eq:fr-dir-est}) and
(\ref{eq:fr-2der-est}) to see that we upper bounds for $\varphi(\lam_1,\lam_2)$
on $W_1$ as follows
\begin{align*}
&\leq C'\left(\frac{\lam_1}{(\lam_1\lam_2)^{r+1/2}}+\frac{\lam_2}{(\lam_1\lam_2)^{r+1/2}}\right)^2 \\
&\phantom{mmmm}\times\left(\frac{1+\lam_{11}}{(\lam_{11}\lam_2)^{r+1/2}}
+\frac{\lam_2}{(\lam_{12}\lam_2)^{r+1/2}}\right)^{2k-2}
\frac{(\lam_1\lam_2)^{1+2r}}{(\lam_1+\lam_2)^{2k-2}} \\
&\leq \frac{C''\lam_1^2\lam_1^{2k-2}\lam_1^{2(1+2r)}}
{\lam_1^{2(2r+1)}\lam_1^{(2r+1)(2k-2)}\lam_1^{2k-2}}
=\frac{C''}{\lam_1^{(2r+1)(2k-2)-2}}
\end{align*}
Applying the same polar coordinate technique as in (\ref{eq:int_W1}) we see that
\begin{equation}\label{eq:int-W1-sing}
\int_{\fW_1}\varphi(\lam_1,\lam_2)\,d(\lam_1,\lam_2)<\infty\text{ if }
\begin{cases}k\geq 2 &\text{when }r\geq 1 \\
k\geq 3 &\text{when }r\geq 0.\end{cases}
\end{equation}

To estimate on $\fW_2$ let us split $\fW_2=\fW_{21}\cup \fW_{22}$ where
\begin{align}
\fW_{21}&=\{(\lam_1,\lam_2)\in \fW_2:\lam_2>c\} \notag \\
\fW_{22}&=\{(\lam_1,\lam_2)\in \fW_2:\lam_2\leq c\} \label{eq:W22}
\end{align}
where $c>0$.
On $\fW_{21}$ we use (\ref{eq:fr-dir-est}) to obtain
\begin{align*}
\varphi(\lam_1,\lam_2)
&\leq C'\left(\frac{\lam_1}{(\lam_1\lam_2)^{r+1/2}}+\frac{\lam_2}{(\lam_1\lam_2)^{r+1/2}}\right)^{2k}
\frac{(\lam_1\lam_2)^{2r+1}}{(\lam_1^2-\lam_2^2)^{2k-2}} \\
&= \frac{C'(\lam_1+\lam_2)^{2k}}{(\lam_1\lam_2)^{(2r+1)(k-1)}(\lam_1^2-\lam_2^2)^{2k-2}} \\
&\leq \frac{C'\lam_1^{2k}}{\lam_1^{(2r+1)(k-1)}\lam_1^{2(2k-2)}}
= \frac{C'}{\lam_1^{(2r+1)(k-1)+2k-2}}
\end{align*}
where we use that $\lam_2\geq c>0$.
On $\fW_{22}$, we additionally use boundedness of $f_r'(x\lam_2)$ and $\lam_2$ to obtain
\begin{align*}
\varphi(\lam_1,\lam_2)
&\leq C''\left(\frac{\lam_1}{\lam_1^{r+1/2}}+\frac{1}{\lam_1^{r+1/2}}\right)^{2k}
\frac{\lam_1^{2r+1}}{(\lam_1^2)^{2k-2}} \\
&=\frac{C''\lam_1^{2k}\lam_1^{2r+1}}{\lam_1^{k(2r+1)}\lam_1^{2(2k-2)}}
=\frac{C''}{\lam_1^{(k-1)(1+2r)+2k-4}}.
\end{align*}
Again we get that
\begin{equation}\label{eq:int-W2-sing}
\int_{\fW_2}\varphi(\lam_1,\lam_2)\,d(\lam_1,\lam_2)<\infty\text{ if }
\begin{cases}k\geq 2 &\text{when }r\geq 1 \\
k\geq 3 &\text{when }r\geq 0.\end{cases}
\end{equation}
Hence (\ref{eq:int-W1-sing}) and (\ref{eq:int-W2-sing}) provide  (\ref{eq:int3}) for $k\geq 2$
when $r\geq 1$, i.e.\ $q>2$; and for $k\geq 3$ when $q=2$.

Now we suppose that $r=0$.  We need only consider the case where $k=2$.
Here $f_0=J_0$, where $J_0'=J_{-1}=-J_1$.  Hence we use (\ref{eq:Jr-asymp}) to see that
\begin{align*}
&\varphi(\lam_1,\lam_2)=
\frac{\begin{vmatrix}\lam_1J_1(x\lam_1) & \lam_2 J_1(x\lam_2) \\
J_0(x\lam_1) & J_0(x\lam_2)\end{vmatrix}^4}
{(\lam_1^2-\lam_2^2)^2}\lam_1\lam_2 \\
&\approx \frac{4|\lam_1\cos(x\lam_1-\frac{3\pi}{4})\cos(x\lam_2-\frac{\pi}{4})
-\lam_2\cos(x\lam_1-\frac{\pi}{4})\cos(x\lam_2-\frac{3\pi}{4})|^4}
{\pi^4x^4\lam_1\lam_2(\lam_1^2-\lam_2^2)^2}
\end{align*}
provided $\lam_2,\lam_2>\!\!> 1$.  Fix $0<\eta<\frac{\pi}{4}$, so we have
\[
|\cos(\eta)|>\frac{1}{\sqrt{2}}\text{ and }
|\cos(\tfrac{\pi}{2}+\eta)|<\frac{1}{\sqrt{2}}.
\]
For positive integer $n$ let $I_n$ denote the interval
\begin{align*}
&\left[\frac{(n+1)\pi-\eta+\frac{3\pi}{4}}{x},\frac{(n+1)\pi+\eta+\frac{3\pi}{4}}{x}\right]
\times \left[\frac{n\pi-\eta+\frac{\pi}{4}}{x},\frac{n\pi+\eta+\frac{\pi}{4}}{x}\right] \\
&=\left[\frac{(n+\frac{1}{2})\pi-\eta+\frac{\pi}{4}}{x},\frac{(n+\frac{1}{2})\pi+\eta+\frac{\pi}{4}}{x}\right] \\ &\quad \times
\left[\frac{(n-\frac{1}{2})\pi-\eta+\frac{3\pi}{4}}{x},\frac{(n-\frac{1}{2})\pi+\eta+\frac{3\pi}{4}}{x}\right].
\end{align*}
Notice for $(\lam_1,\lam_2)\in I_n$, the two descriptions of $I_n$ and $\pi$-periodicity provide
estimates
\begin{align*}
\lam_1|\cos(x\lam_1-\tfrac{3\pi}{4})\cos(x\lam_2-\tfrac{\pi}{4})|
&\geq \frac{(n+\frac{7}{4})\pi+\eta}{x}\cos^2(\eta) \\
\lam_2|\cos(x\lam_1-\tfrac{\pi}{4})\cos(x\lam_2-\tfrac{3\pi}{4})|
&\leq \frac{(n+\frac{1}{4})\pi+\eta}{x}\cos^2(\tfrac{\pi}{2}+\eta)
\end{align*}
while for some constant we get easy estimates
\[
\lam_1,\lam_2,\lam_1^2-\lam_2^2\leq Cn.
\]
Hence as the intervals $I_n$ are pairwise disjoint, each of fixed area $\frac{4\eta}{x^2}$,
and contained in $\fW$ we see that
\begin{align*}
\int_\fW&\varphi(\lam_1,\lam_2)\,d(\lam_1,\lam_2)\geq \sum_{n=1}^\infty \int_{I_n}\varphi(\lam_1,\lam_2)\,d(\lam_1,\lam_2) \\
&\geq  \sum_{n=1}^\infty C''\frac{\left[\frac{(n+\frac{7}{4})\pi+\eta}{x}\cos^2(\eta)
-\frac{(n+\frac{1}{4})\pi+\eta}{x}\cos^2(\tfrac{\pi}{2}+\eta)\right]^4}{n^4}=\infty.
\end{align*}
Hence  (\ref{eq:int3}) fails for $k=2$ when $r=0$, i.e.\ $q=2$.
\end{proof}

\begin{theorem}\label{theo:singular2}
Let $H$ be a singular point of type A in the flat symmetric space
$\SU(2,q)_0/\Ss(\Un(2)\times \Un(q))$.  Then $L^1$-$L^2$ dichotomy
fails when $q\geq 3$.  In fact $\mu_H^{\ast k}\in L^2(\pee)$
only for $k\geq \frac{3}{4}+\frac{q}{2}$ when $q\geq 3$.
\end{theorem}

\begin{proof}
Lemma \ref{lem:inL1} always provides that $\mu_H^{\ast 2}\in L^1(\pee)$.
Hence it remains to check for which $k\geq 2$ (\ref{eq:int3}) holds.  As before, set $r=q-2$.
Since $f_r(0)=\frac{1}{2^r}$ we find that $\varphi_k(\lam,X)$ is proportional to
\[
\varphi(\lam_1,\lam_2)
=\frac{|f_r(x\lam_1)-f_r(x\lam_2)|^{2k}}{(\lam_1^2-\lam_2^2)^{2k-2}}(\lam_1\lam_2)^{1+2r}.
\]
Let $\fW_2=\fW_{21}\cup \fW_{22}$ as in (\ref{eq:W22}).
We consider first
\[
S=[N,\infty)\times [\eps,c]\subset \fW_{22}
\]
where $c>\eps>0$ are chosen so $f_r(x\lam_2)>0$ for $\lam_2\in [0,c]$, and $N>0$ so
$|f_r(x\lam_1)|<f_r(x\lam_2)$ for $(\lam_1,\lam_2)$ in $S$.  Then
on $S$ we have
\begin{equation}\label{eq:asymp-sing}
\varphi(\lam_1,\lam_2)\approx C'\frac{\lam_1^{1+2r}}{\lam_1^{2(2k-2)}}=\frac{C'}{\lam_1^{4k-5-2r}}.
\end{equation}
Thus $\int_S \varphi(\lam_1,\lam_2)\,d(\lam_1,\lam_2)<\infty$ if and only if
$4k-5-2r\geq 2$ so $k\geq \frac{7}{4}+\frac{r}{2}=\frac{3}{4}+\frac{q}{2}$. Since $f_r$
is bounded, we have that (\ref{eq:asymp-sing}) becomes an inequality showing the sufficiency
implication of
\begin{equation}\label{eq:sing-W22}
\int_{\fW_{22}} \varphi(\lam_1,\lam_2)\,d(\lam_1,\lam_2)<\infty \text{ if and only if }
k\geq \frac{7}{4}+\frac{r}{2}
\end{equation}
where necessity is established on $S$.

On $\fW_{21}$ we estimate
\begin{align*}
\varphi(\lam_1,\lam_2)&\leq C'\left(\frac{1}{\lam_1^{r+1/2}}+\frac{1}{\lam_2^{r+1/2}}\right)^{2k}
\frac{(\lam_1\lam_2)^{1+2r}}{(\lam_1^2-\lam_2^2)^{2k-2}} \\
&=C'\frac{(\lam_1^{r+1/2}+\lam_2^{r+1/2})^{2k}}
{(\lam_1\lam_2)^{(2r+1)(k-1)}(\lam_1^2-\lam_2^2)^{2k-2}} \\
&\leq \frac{C''\lam_1^{(2r+1)k}}{(\lam_1\lam_2)^{(2r+1)(k-1)}\lam_1^{2(2k-2)}}
=\frac{C''}{\lam_2^{(2r+1)(k-1)}\lam_1^{4k-5-2r}}.
\end{align*}
Then, assuming $(2r+1)(k-1)>1$, we see that
\begin{align*}
\int_{W_{21}} \varphi(\lam_1,\lam_2)\,d(\lam_1,\lam_2)
&\leq \int_{2c}^\infty\int_c^{\frac{1}{2}\lam_1}\frac{C''}{\lam_2^{(2r+1)(k-1)}\lam_1^{4k-5-2r}}
\,d\lam_2\,d\lam_1 \\
&= \int_{2c}^\infty \left(K-\frac{K'}{\lam_1^{(2r+1)(k-1)-1}}\right)\frac{1}{\lam_1^{4k-5-2r}}\,d\lam_1.
\end{align*}
This integral converges if $4k-5-2r\geq 2$, hence
\begin{equation}\label{eq:sing-W21}
\int_{\fW_{21}} \varphi(\lam_1,\lam_2)\,d(\lam_1,\lam_2)<\infty \text{ if }
k\geq \frac{7}{4}+\frac{r}{2}.
\end{equation}

On $\fW_1$ we have
that $\varphi(\lam_1,\lam_2)$ is equal to
\begin{align*}
&|f_r(x\lam_1)-f_r(x\lam_2)|^2\left|\frac{f_r(x\lam_1)-f_r(x\lam_2)}{\lam_1-\lam_2}\right|^{2k-2}
\frac{(\lam_1\lam_2)^{1+2r}}{(\lam_1+\lam_2)^{2k-2}} \\
&=|f_r(x\lam_1)-f_r(x\lam_2)|^2x|f_r'(x\lam_{11})|^{2k-2}
\frac{(\lam_1\lam_2)^{1+2r}}{(\lam_1+\lam_2)^{2k-2}} \\
&\leq C\left(\frac{1}{\lam_1^{r+1/2}}+\frac{1}{\lam_2^{r+1/2}}\right)^2
\left(\frac{1}{\lam_{11}^{r+1/2}}\right)^{2k-2}
\frac{(\lam_1\lam_2)^{1+2r}}{(\lam_1+\lam_2)^{2k-2}} \\
&\leq \frac{C''\lam_1^{2(1+2r)}}{\lam_1^{2r+1}\lam_1^{(k-1)(2r+1)}\lam_1^{2k-2}}
=\frac{C''}{\lam_1^{(k-2)(2r+1)+2k-2}}.
\end{align*}
Hence we see that
\begin{equation}\label{eq:sing-W1}
\int_{\fW_1} \varphi(\lam_1,\lam_2)\,d(\lam_1,\lam_2)<\infty \text{ if }k\geq 3.
\end{equation}
Notice that when $r\geq 1$, the combination of (\ref{eq:sing-W22}), (\ref{eq:sing-W21}), and
(\ref{eq:sing-W1}) establishes (\ref{eq:int3}) for $k\geq \frac{7}{4}+\frac{r}{2}$.
\end{proof}

\begin{remark} If we are to consider the case $q=2$, i.e.\
$r=0$, it remains to decide if (\ref{eq:int3}) holds for $k=2$.
Here we have
\begin{align*}
\varphi(\lam_1,\lam_2)&=|J_0(x\lam_1)-J_0(x\lam_2)|^4\frac{\lam_1\lam_2}{(\lam_1^2-\lam_2^2)^2} \\
&\approx\left|\frac{1}{\sqrt{\lam_1}}\cos(x\lam_1-\tfrac{\pi}{4})
-\frac{1}{\sqrt{\lam_2}}\cos(x\lam_2-\tfrac{\pi}{4})\right|^4
\frac{\pi^2\lam_1\lam_2}{4x^2(\lam_1^2-\lam_2^2)^2}
\end{align*}
when $\lam_1,\lam_2>\!\!>1$.  The key to understanding this case is to decide
integrability on $\fW_1$, especially where $\lam_1-\lam_2$ is close to $0$.
As of time of writing, we could not determine this.  \end{remark}

\subsection{Rank 3 regular case} The regular rank 3 case admits the same conclusion
as the regular rank 2 case.  However, the techniques are more difficult, which
is why we write this as a different theorem.  In the course of the proof, we point out at
($\dagger$) the difficulties in adapting our methods to higher rank cases.

\begin{theorem}\label{theo:int4-reg}
The flat symmetric space $\SU(3,q)_0/\Ss(\Un(3)\times \Un(q))$ ($q\geq 3$) satisfies the
$L^1$-$L^2$ dichotomy at regular points.
\end{theorem}

\begin{proof}
As with the proof of Theorem \ref{theo:int3-reg}, it suffices to verify that
(\ref{eq:int3}) holds for $k\geq 2$.  The integrand is proportional to
\[
\varphi(\lam)=\frac{\begin{vmatrix}f_r(x_1\lam_1) & f_r(x_1\lam_2) & f_r(x_1\lam_3) \\
f_r(x_2\lam_1) & f_r(x_2\lam_2) & f_r(x_2\lam_3) \\
f_r(x_3\lam_1) & f_r(x_3\lam_2) & f_r(x_3\lam_3)
\end{vmatrix}^{2k}}
{\prod_{1\leq j<k\leq 3}(\lam_j^2-\lam_k^2)^{2k-2}}
(\lam_1\lam_2\lam_3)^{1+2r}
\]
where $r=q-3$ and $f_r$ is defined in (\ref{eq:fr}).
We let $B$ be a ball of large radius and decompose
$\fW=(\fW\cap B)\cup\fW_1\cup \fW_2\cup \fW_3$
where
\begin{align*}
\fW_1&=\{(\lam_1,\lam_2,\lam_3)\in \fW\setminus B:\lam_3\geq \tfrac{1}{2}\lam_1\} \\
\fW_2&=\{(\lam_1,\lam_2,\lam_3)\in \fW\setminus B:\lam_2\geq \tfrac{1}{2}\lam_1>\lam_3 \} \\
\fW_3&=\{(\lam_1,\lam_2,\lam_3)\in \fW\setminus B: \tfrac{1}{2}\lam_1>\lam_2\}
\end{align*}

On $\fW_1$ we consider estimates on a factor of $\varphi(\lam)$, using mean value theorem techniques.  We have
\begin{align}
&\frac{\begin{vmatrix}f_r(x_1\lam_1) & f_r(x_1\lam_2) & f_r(x_1\lam_3) \\
f_r(x_2\lam_1) & f_r(x_2\lam_2) & f_r(x_2\lam_3) \\
f_r(x_3\lam_1) & f_r(x_3\lam_2) & f_r(x_3\lam_3)
\end{vmatrix}^{2k-2}}
{\prod_{1\leq j<k\leq 3}(\lam_j-\lam_k)^{2k-2}} \notag \\
&=\frac{\begin{vmatrix}\frac{f_r(x_1\lam_1)-f_r(x_1\lam_2)}{\lam_1-\lam_2} &
\frac{f_r(x_1\lam_2)-f_r(x_1\lam_3)}{\lam_2-\lam_3} & f_r(x_1\lam_3) \\
\frac{f_r(x_2\lam_1)-f_r(x_2\lam_2)}{\lam_1-\lam_2} &
\frac{f_r(x_2\lam_2)-f_r(x_2\lam_3)}{\lam_2-\lam_3} & f_r(x_2\lam_3) \\
\frac{f_r(x_3\lam_1)-f_r(x_3\lam_2)}{\lam_1-\lam_2} &
\frac{f_r(x_3\lam_2)-f_r(x_3\lam_3)}{\lam_2-\lam_3} & f_r(x_3\lam_3)
\end{vmatrix}^{2k}}
{(\lam_1-\lam_3)^{2k-2}} \notag \\
&=
\frac{\begin{vmatrix}x_1f_r'(x_1\lam_{11}) & x_1f_r'(x_1\lam_{21}) & f_r(x_1\lam_3) \\
x_2f_r'(x_2\lam_{12}) & x_2f_r'(x_2\lam_{22}) & f_r(x_2\lam_3) \\
x_3f_r'(x_3\lam_{13}) & x_3f_r'(x_3\lam_{23}) & f_r(x_3\lam_3)
\end{vmatrix}^{2k-2}}{(\lam_1-\lam_3)^{2k-2}}
\label{eq:W1_factor}
\end{align}
where $ \lam_j>\lam_{jk}>\lam_{j+1}$ for $j=1,2, k=1,2,3$.  We again apply mean value
theorem technique to get for this factor
\begin{align*}
&=\begin{vmatrix}x_1\frac{f_r'(x_1\lam_{11})-f_r'(x_1\lam_{21})}{\lam_1-\lam_3}
& x_1f_r'(x_1\lam_{21}) & f_r(x_1\lam_3) \\
x_2\frac{f_r'(x_2\lam_{12})-f_r'(x_2\lam_{22})}{\lam_2-\lam_3}
& x_2f_r'(x_2\lam_{22}) & f_r(x_2\lam_3) \\
x_3\frac{f_r'(x_3\lam_{13}) -f_r'(x_3\lam_{23}) }{\lam_1-\lam_3}
& x_3f_r'(x_3\lam_{23}) & f_r(x_3\lam_3)
\end{vmatrix}^{2k-2}  \\
&=\begin{vmatrix}x_1^2f_r''(x_1\lam'_{11}) \frac{\lam_{11}-\lam_{21}}{\lam_1-\lam_3}
& x_1f_r'(x_1\lam_{21}) & f_r(x_1\lam_3) \\
x_2^2f_r''(x_2\lam'_{12}) \frac{\lam_{12}-\lam_{22}}{\lam_1-\lam_3}
& x_2f_r'(x_2\lam_{22}) & f_r(x_2\lam_3) \\
x_3^2f_r''(x_3\lam'_{13}) \frac{\lam_{13}-\lam_{23}}{\lam_1-\lam_3}
& x_3f_r'(x_3\lam_{23}) & f_r(x_3\lam_3)
\end{vmatrix}^{2k-2}
\end{align*}
where $\lam_1>\lam_{1k}>\lam'_{1k}>\lam_{2k}>\lam_3$ for $k=1,2,3$.
We then apply  (\ref{eq:fr-est}), (\ref{eq:fr-dir-est}), (\ref{eq:fr-2der-est})
and the fact that each quotient $|\frac{\lam_{1k}-\lam_{2k}}{\lam_1-\lam_3}|\leq 1$,
to obtain upper bound for this factor
\[
\leq\frac{C'}{\lam_3^{(r+1/2)3(2k-2)}}\leq \frac{C''}{\lam_1^{(1+2r)3(k-1)}}
\]
where constants may depend on $k$, but not $\lambda$.
[($\dagger$) It is at this point we have difficulty moving this argument to
$\SU(n,q)$ with $q\geq n>3$.  At 2 steps we can arrange the ordering
$ \lam_1<\lam_{1k}<\lam_{2k}<\lam_3$ which allows for control of
quotients $|\frac{\lam_{1k}-\lam_{2k}}{\lam_1-\lam_3}|$.  At 3 steps this argument
breaks down.] Thus we find on $\fW_1$ that
\begin{align*}
\varphi(\lam)&\leq
C'''\frac{1}{(\lam_1\lam_2\lam_3)^{(r+1/2)2}}\cdot\frac{1}{\lam_1^{(1+2r)3(k-1)}}\cdot\frac{(\lam_1\lam_2\lam_3)^{1+2r}}
{\prod_{1\leq j<k\leq 3}(\lam_j+\lam_k)^{2k-2}} \\
&\leq \frac{C'''}{\lam_1^{(1+2r)3(k-1)}\lam_1^{2(2k-2)}\lam_2^{2k-2}}
=\frac{C'''}{\lam_1^{3(1+2r)(k-1)+6k-6}}
\end{align*}
Hence, by
setting up spherical coordinates with $\lam_1=\rho\cos\theta\sin\phi$
for a region containing $W_1$, we see that
\[
\int_{\fW_1}\varphi(\lam)\,d\lam
\leq K\int_{\frac{\pi}{4}}^{\frac{\pi}{2}}\int_0^{\frac{\pi}{4}}\int_1^\infty
\frac{\rho^2\,d\rho\,d\theta\,d\phi}{\rho^{3(1+2r)(k-1)+4k-4}}<\infty
\]
provided $k\geq 2$.

On $\fW_2$ we begin with (\ref{eq:W1_factor}) for the same factor.
Since each $f_r(x_k\lam_3)$ is bounded, $\lam_2\geq\frac{1}{2}\lam_1$,
and $\frac{1}{2}\lam_1>\lam_3$ so $\lam_1-\lam_3\geq \frac{1}{2}\lam_1$,
we apply  (\ref{eq:fr-est}) and (\ref{eq:fr-dir-est}) to get
estimate for this factor
\[
\leq \frac{C'}{(\lam_1\lam_2)^{(r+1/2)(2k-2)}\lam_1^{2k-2}}\leq \frac{C''}{\lam_1^{(1+2r)(2k-2)+2k-2}}.
\]
We again  apply  (\ref{eq:fr-est}) and the fact that $f_r(x_k\lam_3)$ is bounded to get
\begin{align*}
\varphi(\lam)&\leq
\frac{1}{(\lam_1\lam_2)^{(r+1/2)2}}\cdot\frac{C''}{\lam_1^{(1+2r)(2k-2)+2k-2}}\cdot\frac{(\lam_1\lam_2\lam_3)^{1+2r}}{\prod_{1\leq j<k\leq 3}(\lam_j+\lam_k)^{2k-2}} \\
&\leq \frac{C'''\lam_3^{1+2r}}{\lam_1^{(1+2r)(2k-2)+2k-2}\lam_1^{2(2k-2)}\lam_2^{2k-2}}
\leq \frac{C'''}{\lam_1^{(1+2r)(2k-3)+8k-8}}
\end{align*}
Again, by setting up spherical coordinates
for a region containing $W_2$, we see that
\[
\int_{\fW_2}\varphi(\lam)\,d\lam
\leq K\int_{\frac{\pi}{2}}^{\frac{\pi}{8}}\int_{\frac{\pi}{8}}^{\frac{\pi}{4}}\int_1^\infty
\frac{\rho^2\,d\rho\,d\theta\,d\phi}{\rho^{(1+2r)(2k-3)+8k-8}}<\infty
\]
provided $k\geq 2$.

To estimate over $\fW_3$ we need to consider the fact that $\lam_2$ and $\lam_3$
may be simultaneously small, and close together.  Define
\[
g_r(t)=\sum_{l=0}^\infty \frac{(-1)^l}{2^{2l+r}l!(r+l)!}t^l
\text{ so }g_r(t^2)=f_r(t).
\]
Hence for $t>0$ we have $g_r(t)=f_r(\sqrt{t})$, so $g_r'(t)=\frac{f_r'(\sqrt{t})}{2\sqrt{t}}$,
and it follows that $g_r'$ is bounded.  Now we estimate a factor of $\varphi(\lam)$ as follows
\begin{align*}
&\frac{\begin{vmatrix}f_r(x_1\lam_1) & f_r(x_1\lam_2) & f_r(x_1\lam_3) \\
f_r(x_2\lam_1) & f_r(x_2\lam_2) & f_r(x_2\lam_3) \\
f_r(x_3\lam_1) & f_r(x_3\lam_2) & f_r(x_3\lam_3)
\end{vmatrix}^{2k-2}}{(\lam_2^2-\lam_3^2)^{2k-2}} \\
&=\begin{vmatrix}f_r(x_1\lam_1)
& \frac{g_r(x_1^2\lam_2^2)-g_r(x_1^2\lam_3^2)}{\lam_2^2-\lam_3^2} & f_r(x_1\lam_3) \\
f_r(x_2\lam_1)
& \frac{g_r(x_2^2\lam_2^2)-g_r(x_2^2\lam_3^2)}{\lam_2^2-\lam_3^2} & f_r(x_2\lam_3) \\
f_r(x_3\lam_1)
&\frac{g_r(x_3^2\lam_2^2)-g_r(x_3^2\lam_3^2)}{\lam_2^2-\lam_3^2} & f_r(x_3\lam_3)
\end{vmatrix}^{2k-2} \\
&=\begin{vmatrix}f_r(x_1\lam_1) & x_1^2g_r'(x_1^2\mu_1) & f_r(x_1\lam_3) \\
f_r(x_2\lam_1) & x_2^2g_r'(x_2^2\mu_2) & f_r(x_2\lam_3) \\
f_r(x_3\lam_1) &  x_3^2g_r'(x_3^2\mu_3) & f_r(x_3\lam_3)
\end{vmatrix}^{2k-2}
\end{align*}
where $\lam_2^2>\mu_k>\lam_3^2$ for $1\leq k\leq 3$.
Then (\ref{eq:fr-est}) and the boundedness of $g_r'(x_k\mu_k)$ and $f_r(x_k\lam_3)$
provide estimate for this factor of
\[
\leq \frac{C'}{\lam_1^{(r+1/2)(2k-2)}}=\frac{C'}{\lam_1^{(1+2r)(k-1)}}.
\]
Now we decompose $\fW_3=\fW_{31}\cup \fW_{32}$ where
\begin{align*}
\fW_{31}&=\{\lam\in \fW_3:\lam_3>c\} \\
\fW_{32}&=\{\lam\in \fW_3:\lam_3\leq c\}
\end{align*}
where $c>0$.
Then as each $f_r(x_j\lam_3)$ is bounded on $\fW_{31}$, and we can suppose $c$ is
large enough to support estimate (\ref{eq:fr-est}),
we obtain
\begin{align*}
\varphi(\lam)&\leq C''\frac{1}{(\lam_1\lam_2\lam_3)^{(r+1/2)2}}\cdot\frac{1}{\lam_1^{(1+2r)(k-1)}}
\cdot \frac{(\lam_1\lam_2\lam_3)^{1+2r}}{(\lam_1^2-\lam_2^2)^{2k-2}(\lam_1^2-\lam_3^2)^{2k-2}} \\
&\leq \frac{C''}{\lam_1^{(1+2r)(k-1)+8k-8}}.
\end{align*}
As each $f_r(x_j\lam_k)$ is bounded on $W_{32}$, for $k=2,3$, as is $\lam_3$, we obtain
the estimate
\begin{align*}
\varphi(\lam)&\leq C'''\frac{1}{\lam_1^{(r+1/2)2}}\cdot\frac{1}{\lam_1^{(1+2r)(k-1)}}
\cdot \frac{(\lam_1\lam_2)^{1+2r}}{(\lam_1^2-\lam_2^2)^{2k-2}(\lam_1^2-\lam_3^2)^{2k-2}} \\
&\leq \frac{C'''}{\lam_1^{(1+2r)(k-2)+8k-8}}\leq \frac{C'''}{\lam_1^{(1+2r)(k-1)+8k-8}}.
\end{align*}
Again, by setting up spherical coordinates
for a region containing $W_3$, we see that
\[
\int_{\fW_3}\varphi(\lam)\,d\lam
\leq K\int_{\frac{3\pi}{8}}^{\frac{\pi}{2}}\int_0^{\frac{\pi}{8}}\int_1^\infty
\frac{\rho^2\,d\rho\,d\theta\,d\phi}{\rho^{(1+2r)(k-1)+8k-8}}<\infty
\]
provided $k\geq 2$.

All of our estimates provide the desired result.
\end{proof}

\end{document}